\documentclass{article}
\usepackage[russian]{babel}
\usepackage{amsmath}
\usepackage{amsthm}
\usepackage{amsfonts}

\newcommand{\im}{\operatorname{Im}}

\theoremstyle{plain}
\newtheorem{theorem}{Теорема}[]

\theoremstyle{definition}

\theoremstyle{definition}

\theoremstyle{definition}
\newtheorem{lemma}{Лемма}
\theoremstyle{remark}
\newtheorem{cor}[lemma]{Следствие}
\theoremstyle{definition}

\begin{document}
\mag1300
\sloppy

{\Large
\begin{center} Stability in  the inverse resonance  problem for
the Schr\" odinger operator
 \end{center}}

\bigskip
\centerline{V.~L.~Geynts and A.~A.~Shkalikov}\footnote{This work is supported by Russian Science Foundation (RNF) under grant No 17-11-01215.}

\bigskip
\medskip

\centerline{{\bf Abstract}}
\bigskip
We work with the Schr\" odinger equation
\begin{equation*}
H_q y = -y'' + q(x)y = z^2y, \  x\in [0,\infty),
\end{equation*}
where $q\in L_1((0,\infty), xdx)$,
and asssume that the corresponding operator  $H_q$ is defined by the Dirihlet condition  $y(0) = 0$ The function $\psi(z) = y(0,z)$ where $y(x,z)$ is the Jost solution of the above equation is analytic in the whole complex plane,
provided that the support of the potential $q$ is  finite. The zeros of $\psi$ are called the resonances. It is known that $q$ is uniquely  determined by the sequence of resonances. Using only finitely many resonances lying in the disk
$|z|\le r$ we can recover the potential $q$ with accuracy $\varepsilon(r)\to 0$  as $r \to \infty$. The main result of the paper is the estimate  $\varepsilon(r) \le Cr^{-\alpha}$ with some constants $C$ and $\alpha>0$ which are defined by  a priori information about the potential $q$.


\bigskip
\medskip

{\bf 1. Введение.}

Понятие резонанса оператора Шредингера --- одно из фундаментальных в квантовой механике. С физической точки зрения, собственному значению оператора соответствует стационарное состояние частицы, а резонансу $z_0\in \mathbb{C}$ --- квазистационарное, в котором частица находится в течение времени, обратно пропорционального $|\im(z_0)|$.
Научно-популярный обзор о резонансах содержится в статье \cite{ZworskiResonances}, обстоятельное изложение --- в книге \cite{ZworskiDyatlov}.

Среди обратных задач квантовой теории рассеяния значительный интерес представляет задача восстановления потенциала оператора Шредингера по данному конечному числу <<ближних>> резонансов и собственных значений, а также вопрос об устойчивости такого восстановления. Численным алгоритмам посвящены работы \cite{RundellSacks}, \cite{Aktosun}, задаче об устойчивости --- работы \cite{KorotyaevInvProblem}, \cite{KorotyaevStability} в случае полного набора собственных значений и резонансов и \cite{Marletta}, \cite{Bledsoe} в случае конечного набора. В данной статье будет получено усиление и обобщение результатов работы \cite{Marletta} для случая потенциалов с компактным носителем. 

Рассматриваем стационарное уравнение Шредингера на полуоси
\begin{equation}\label{eq_Schroedinger}
H_q y = -y'' + q(x)y = z^2y, \  x\in [0,\infty),
\end{equation} где $q\in L_1((0,\infty), xdx)$,
и соответствующий оператор $H_q$ с граничным условием Дирихле $y(0) = 0$.
В дальнейшем $\mathbb{C}_+$ --- открытая верхняя полуплоскость, $B(a, r)$ --- открытый круг радиуса $r$ с центром в $a\in \mathbb{C}$, $Z(f)$ --- множество нулей комплекснозначной функции $f$, $\textbf{supp}(f)$ --- ее носитель, а $\Vert\cdot\Vert_p$ --- нормы пространств $L_p$, $1\le p\le \infty$.

При $z\in \overline{\mathbb{C}}_+$ функция $y_q(\cdot, z)\colon [0,\infty)\to \mathbb{C}$ --- единственное решение уравнения \eqref{eq_Schroedinger},
удовлетворяющее равенству
$$
\lim\limits_{x\to \infty}\frac{y_q(x, z)}{e^{izx}} = 1;
$$
эта функция (см. \cite{Marchenko}) называется решением Йоста уравнения \eqref{eq_Schroedinger}.
Функция $\psi_{q}(\cdot) = y_{q}(0, \cdot)$ называется функцией Йоста и является голоморфной в $\mathbb{C}_+$ и непрерывной в $\overline{\mathbb{C}}_+$.
Если потенциал $q$ убывает суперэкспоненциально в смысле $\int\limits_{0}^{\infty}|q(x)|e^{x^{\gamma}}<\infty$ для некоторого $\gamma>1$, то $\psi_q$ продолжается в $\mathbb{C}$ как целая функция. Ее нули в области $\mathbb{C}_+$ являются корнями из собственных значений оператора $H_q$, а нули, лежащие в $\overline{\mathbb{C}}_-$, называются резонансами оператора $H_q$.

Пусть $q_j\in L_1((0, \infty), xdx)$, $j = 1, 2$; для единообразия, обозначим $q_0 \equiv 0$, тогда $y_{q_0}(x, z) = e^{ixz}$. Через $K_{ij}$ ($i, j\in\{0,1,2\}$) обозначим ядро оператора преобразования, переводящего  функцию $y_{q_i}(\cdot, z)$ в функцию $y_{q_j}(\cdot, z)$, т.е.
$$
y_{q_j}(x, z) = y_{q_i}(x, z) + \int\limits_{x}^{\infty}K_{ij}(x, t)y_{q_i}(t, z)dt,
$$
и, в частности,
$$
y_{q_j}(x, z) = e^{ixz} + \int\limits_{x}^{\infty}K_{0j}(x, t)e^{izt}dt, \ j\in \left\{ 1, 2\right\}.
$$


Задача состоит в том, чтобы, владея некоторой априорной информацией о суперэкспоненциально убывающих потенциалах $q_1$ и $q_2$, из близости нулей функций Йоста для $q_1$ и  $q_2$ внутри достаточно большого круга сделать вывод о близости $q_1$ и $q_2$ в некоторой метрике. В данной работе предполагаем, что носители $q_1$ и $q_2$ компактны.\bigskip

{\bf 2. Предварительные оценки.}
Для потенциала $q\in L_1((0,\infty), xdx)$ введем невозрастающие функции
\begin{eqnarray}
&\sigma_q(x) := \int\limits_{x}^{\infty}|q(s)|ds = \Vert q(\cdot+x)\Vert_{L_1(dx)},  \\
&\overline{\sigma}_q(x):=\int\limits_{x}^{\infty}\sigma_q(s)ds = \int\limits_{x}^{\infty}(t-x)|q(t)|dt = \Vert q(\cdot + x)\Vert_{L_1(xdx)}.
\end{eqnarray}

\begin{theorem}[\cite{Marchenko}]
Пусть $q_j\in L_1((0, \infty), xdx)$, $j = 1, 2$. Тогда
ядро $K_{12}$ удовлетворяет интегральному уравнению
\begin{equation}
\label{eq_Kernel_12}
K_{12}(x,t) = \frac{1}{2}\int\limits_{\frac{t+x}{2}}^{\infty}(q_2(s)-q_1(s))ds + H_{12}(x, t),
\end{equation}
где
\begin{equation}
\label{eq_H_12}
H_{12}(x, t) = \int\limits_{0}^{\frac{t-x}{2}}\, du\int\limits_{\frac{t+x}{2}}^{\infty}\, dv  \ (q_2(v - u)-q_1(v + u))K_{12}(v - u, v + u).
\end{equation}
\end{theorem}

В частности,
\begin{equation}
K_{12}(x,x) = \frac{1}{2}\int\limits_{x}^{\infty}(q_2(s) - q_1(s))ds.
\end{equation}

Для удобства применим преобразование координат $$
u = \frac{t-x}{2}, \ v = \frac{t + x}{2}$$
и для функции $F$ переменных $(x, t)$ будем обозначать $\widetilde{F}(u, v) = F(v - u, v + u)$. Заметим, что если $\textbf{supp}(q_j)\subset [0, a]$, $j = 1, 2$, то
$\textbf{supp}(K_{12})\subset \{(x, t): 0\le x\le a, \ x\le t\le 2a-x\}$ и
$\textbf{supp}(\widetilde{K}_{12})\subset \{(u, v): 0\le u\le v\le a\}$.

\begin{lemma}
Пусть $q_1, q_2\in L_1((0, \infty), xdx)$. Тогда верны оценки
\begin{equation}\label{eq_KernelEstimates}
\begin{gathered}
|K_{12}(x,t)|\le\frac{1}{2}\sigma_{q_2-q_1}  \left(\frac{t+x}{2}\right)M_{12}(x, t);\\
\max\left\{\left| \frac{\partial H_{12}}{\partial x}(x, t)\right|, \left| \frac{\partial H_{12}}{\partial t}(x, t)\right|\right\} \le \\
\le \frac{1}{4}\left( 2\sigma_{q_2}(x) + \sigma_{q_1}\left(\frac{t+x}{2}\right) - \sigma_{q_2}\left(\frac{t+x}{2}\right)\right) \sigma_{q_2 - q_1}\left(\frac{t+x}{2}\right) M_{12}(x, t),
\end{gathered}
\end{equation}
где
$$
M_{12}(x, t) = \exp\left( \overline{\sigma}_{q_2}(x) - \overline{\sigma}_{q_2}\left(\frac{t+x}{2}\right) + \overline{\sigma}_{q_1}\left(\frac{t+x}{2}\right) - \overline{\sigma}_{q_1}(t)\right), \
$$
\end{lemma}

В частности, если $\textbf{supp}(q_j)\subset [0, a]$ и $\int\limits_{0}^a |q_j(s)|ds \le Q_1$, то при $x\ge 0$
$$
\sigma_{q_j}(x)\le Q_1, \ \overline{\sigma}_{q_j}(x)\le (a - x)_+Q_1\le aQ_1, \ \sigma_{q_2 - q_1}(x)\le 2Q_1,
$$
при $0\le x \le t$
$$
\max\left\{\overline{\sigma}_{q_2}(x) - \overline{\sigma}_{q_2}\left(\frac{t+x}{2}\right), \ \overline{\sigma}_{q_1}\left(\frac{t+x}{2}\right) - \overline{\sigma}_{q_1}(t)\right\}\le \frac{t-x}{2}Q_1,
$$
откуда следует, что при $0\le x\le t\le 2a$
$$
M_{12}(x, t)\le \exp\left((t-x)Q_1\right)\le e^{2aQ_1}, \ M_{0j}(x, t)\le \exp\left(\frac{t-x}{2}Q_1\right)\le e^{aQ_1},
$$
$$
|K_{12}(x, t)|\le Q_1\exp\left((t-x)Q_1\right)\le Q_1e^{2aQ_1},
$$
$$
\max\left\{|K_{0j}(x, t)|, |K_{j0}(x, t)|\right\}\le\frac{Q_1}{2}\exp\left(\frac{t-x}{2}Q_1\right) \le \frac{Q_1}{2}e^{aQ_1},
$$
$$
\max\left\{\left| \frac{\partial H_{12}}{\partial x}(x, t)\right|, \left| \frac{\partial H_{12}}{\partial t}(x, t)\right|\right\} \le \frac{3}{2}Q_1^2\exp\left((t-x)Q_1\right)\le \frac{3}{2}Q_1^2e^{2aQ_1},
$$
$$
\max\left\{\left| \frac{\partial H_{0j}}{\partial x}(x, t)\right|, \left| \frac{\partial H_{0j}}{\partial t}(x, t)\right|\right\} \le \frac{Q_1^2}{2}\exp\left(\frac{t-x}{2}Q_1\right)\le \frac{Q_1^2}{2}e^{aQ_1},
$$

{\bf Доказательство.}

Частный случай (при $q_1 \equiv 0$) уравнения \eqref{eq_Kernel_12} и оценок \eqref{eq_KernelEstimates} рассмотрен в \cite{Marchenko}.


Приведем доказательство последнего неравенства.
Запишем \eqref{eq_H_12} в виде
$$
\widetilde{H}_{12}(u, v) =
\int\limits_{0}^u dr\int\limits_{v}^{\infty} ds  \left( q_2(s-r) - q_1(s+r)\right) \widetilde{K}_{12}(r, s).
$$
Тогда
$$
\frac{\partial \widetilde{H}_{12}}{\partial u}(u, v) = \int\limits_{v}^{\infty} ds  (q_2(s - u) - q_1(s + u))\widetilde{K}_{12}(u, s),
$$
$$
\frac{\partial \widetilde{H}_{12}}{\partial v}(u, v) = -\int\limits_{0}^u ds  \left( q_2(v - s) - q_1(v + s)\right)\widetilde{K}_{12}(s, v).
$$
Легко видеть, что функция $$\widetilde{M}_{12}(u, v) = M_{12}(v - u, v + u) = \exp\left( \overline{\sigma}_{q_2}(v - u) - \overline{\sigma}_{q_2}(v) + \overline{\sigma}_{q_1}(v) - \overline{\sigma}_{q_1}(v + u)\right)$$ является неубывающей относительно $u$ и невозрастающей относительно $v$. Используя этот факт, получаем оценки
$$
\left|\frac{\partial \widetilde{H}_{12}}{\partial u}(u, v)\right|\le \frac{1}{2}\int\limits_{v}^{\infty}ds  |q_2(s - u) - q_1(s + u)|\sigma_{q_2 - q_1}(s)\widetilde{M}_{12}(u, s)\le$$
$$\le \frac{1}{2}\left( \sigma_{q_2}(v - u) + \sigma_{q_1}(v + u)\right)\sigma_{q_2 - q_1}(v)\widetilde{M}_{12}(u, v);
$$
$$
\left| \frac{\partial \widetilde{H}_{12}}{\partial v}(u, v)  \right|
\le \frac{1}{2}\int\limits_{0}^u ds  |q_2(v - s) - q_1(v + s)|\sigma_{q_2 - q_1}(v)\widetilde{M}_{12}(s, v)\le
$$
$$\le\frac{1}{2}\left(\sigma_{q_2}(v - u) -\sigma_{q_2}(v)+ \sigma_{q_1}(v) - \sigma_{q_1}(v + u)\right)\sigma_{q_2 - q_1}(v)\widetilde{M}_{12}(u, v).
$$
Поскольку
$$
\frac{\partial H_{12}}{\partial x}(x, t) = \frac{1}{2}\left(\frac{\partial \widetilde{H}_{12}}{\partial v} - \frac{\partial \widetilde{H}_{12}}{\partial u}\right)\left(\frac{t-x}{2}, \frac{t + x}{2}\right)
$$
и
$$
\frac{\partial H_{12}}{\partial t}(x, t) = \frac{1}{2}\left(\frac{\partial \widetilde{H}_{12}}{\partial v} + \frac{\partial \widetilde{H}_{12}}{\partial u}\right)\left(\frac{t-x}{2}, \frac{t + x}{2}\right),
$$
то, складывая предыдущие неравенства, получаем
требуемое.

\begin{lemma} Пусть  для некоторого $\gamma > 1$ выполнено
$$
\int\limits_{0}^{\infty}|q(x)|e^{x^\gamma}dx < \infty.
$$
Тогда функция Йоста $$\overline{\mathbb{C}}_+\ni z\mapsto\psi_q(z):= y_q(0, z) = 1 + \int\limits_{0}^{\infty}K_q(0, t)e^{izt}dt$$ продолжается до целой функции порядка, не превосходящего $\frac{\gamma}{\gamma-1}$, и конечного типа. В частности, если $q$ имеет компактный носитель, то $\psi_q$ --- целая функция порядка $\le 1$.
\end{lemma}

\bigskip

{\bf 3. Постановка задачи. Априорные условия на потенциалы.}

В дальнейшем предполагаем, что носители потенциалов $q_1, q_2$ содержатся в отрезке $[0, a]$, $a > 0$, и в качестве априорных
условий примем следующие: $\Vert q_j\Vert_1\le Q_1$ для некоторого $Q_1 > 0$, $j = 1, 2$, и
$$
\Vert q_2 - q_1\Vert_p\le D_p
$$
для некоторых $p\in (1, \infty]$, $D_p > 0$.

Обозначим через $N_{\psi}(R)$ количество нулей (с учетом кратности) целой функции $\psi$ в круге $\overline{B}(0, R)$.
Предположим, что в круге $\overline{B}(0, R)$ находится одинаковое количество корней функций Йоста $\psi_1 $ и $\psi_2$, причем кратности нуля как корня функций $\psi_1$ и $\psi_2$ совпадают,
т.е. $$N_{\psi_1}(0) = N_{\psi_2}(0) =: N(0), \ N_{\psi_1}(R) = N_{\psi_2}(R), \ Z\left(\psi_j\right)\cap \overline{B}(0, R)\setminus\{0\} = \left\{z_n^{(j)}\right\}_{n=1}^{N(R)}, $$ $$N(R) := N_{\psi_j}(R) - N(0), \ j = 1, 2,
$$
и нули $\left\{z_n^{(j)}\right\}$ упорядочены так, что выполнено неравенство
\begin{equation}\label{eq_nearness}
\left|\frac{1}{z_n^{(2)}} - \frac{1}{z_n^{(1)}}\right| < \varepsilon, \ n = 1,...,N(R).
\end{equation}

Требуется оценить
норму
$$
\max_{0 \le x \le a}\left|\int\limits_{x}^{a}(q_2(s) - q_1(s))ds\right|
$$
и при некоторых дополнительных условиях, которые будут сформулированы позже, почти всюду величину
$
|q_2(x) - q_1(x)|
$.

\bigskip

{\bf 4. Метод получения оценок.}
Поскольку
$$ \int\limits_{x}^{a}(q_2(s) - q_1(s))ds = 2K_{12}(x, x),
$$
то для оценки $\left|\int\limits_{x}^{a}(q_2(s) - q_1(s))ds\right|$ достаточно оценить ядро $K_{12}(x, t)$ при $0\le x\le a$, $x\le t\le 2a-x$.
\begin{lemma}
Функция $\widetilde{K}_{12}(u, v)$ удовлетворяет уравнению
\begin{equation}
\label{eq_K12_integral_equation}
\widetilde{K}_{12}(u, v) = K_{12}(0, 2v) + \int\limits_{u}^vdr\int\limits_{v}^{a}ds (q_1(s + r) - q_2(s - r))\widetilde{K}_{12}(r, s).
\end{equation}
\end{lemma}
{\bf Доказательство.}
Утверждение следует из равенства
\begin{equation}
\frac{\partial^2 \widetilde{K}_{12}}{\partial u\partial v}(u, v) = (q_1(v + u) - q_2(v - u))\widetilde{K}_{12}(u, v),
\end{equation}
частный случай которого при $q_1\equiv 0$ приведен в \cite{Marchenko}.

Для того, чтобы оценить решение этого уравнения $\widetilde{K}_{12}(u, v)$, требуется оценить функцию $K_{12}(0, t)$, $0 < t <  2a$.
\begin{lemma}
Функция $K_{12}(0, t)$ связана с разностью $(\psi_2 - \psi_1)|_{\mathbb{R}}$ соотношениями:
\begin{equation}
\label{eq_K12_K1_K_2}
K_{12}(0, t) = K_{02}(0, t) - K_{01}(0, t) + \int\limits_{0}^{t}\left( K_{02}(0, s) - K_{01}(0, s)\right) K_{10}(s, t)ds, \ t\in[0, 2a],
\end{equation}
и для всех $t\in [0, 2a]$
\begin{equation}
\label{eq_K2_minus_K1_Jost}
K_{02}(0, t) - K_{01}(0, t) = \lim\limits_{A\to \infty} \frac{1}{2\pi}\int\limits_{-A}^{A}\left( \psi_2(z) - \psi_1(z)\right) e^{-izt}dz .
\end{equation}
\end{lemma}
{\bf Доказательство.}
Это утверждение доказано в \cite{Marletta}. Поскольку $K_{0j}(0, \cdot)\in AC[0, 2a]$, уравнение \eqref{eq_K2_minus_K1_Jost} является обращением  уравнения
$$
\psi_2(z) - \psi_1(z) = \int\limits_{0}^{2a}\left( K_{02}(0, t) - K_{01}(0, t)\right) e^{izt}dt, \ z\in \mathbb{R}.
$$

Таким образом, оценка интеграла
\begin{equation}
\label{eq_Integral}
\lim\limits_{A\to \infty}\int\limits_{-A}^{A}\left(\psi_2(x) - \psi_1(x)\right) e^{-ixt}dx
\end{equation}
влечет оценку $K_{12}(0, t)$, которая при подстановке в итерационный процесс для решения уравнения \eqref{eq_K12_integral_equation} дает требуемую оценку для $K_{12}(x, t)$.

\bigskip

{\bf 5. Оценка интеграла \eqref{eq_Integral}.}

\begin{lemma}\label{lemma_vp_estimates}
Если $u\in \mathbb{R}\setminus\{0\}$ и $\rho > 0$, то для $A > \rho$
\begin{equation}
\label{eq_int_sinc_tr}
\left|\int\limits_{\rho < |x| < A}\frac{e^{-ixu}}{x}dx\right| \le 2\pi\min\left( 1, \frac{3/\pi}{\rho|u|}\right)
\end{equation}
и
\begin{equation}\label{eq_int_sinc}
\left|\lim\limits_{A\to\infty}\int\limits_{\rho < |x| < A}\frac{e^{-ixu}}{x}dx\right| \le  \pi\min\left( 1, \frac{4/\pi}{\rho|u|}\right).
\end{equation}
\end{lemma}
\begin{proof}
Имеем
$$
\int\limits_{\rho < |x| < A}\frac{e^{-ixu}}{x}dx = -2 i\int\limits_{\rho}^A\frac{\sin(xu)}{x}dx = 2i\left(\frac{\cos(uA)}{uA} - \frac{\cos(u\rho)}{u\rho}\right) + 2i \ \mathrm{sign}(u)\int\limits_{|u|\rho}^{|u|A}\frac{\cos(x)}{x^2}dx,
$$
отсюда
$$
\left|\int\limits_{\rho < |x| < A}\frac{e^{-ixu}}{x}dx\right|\le 6\frac{\rho^{-1}}{|u|}, \ \left|\lim\limits_{A\to\infty }\int\limits_{\rho < |x| < A}\frac{e^{-ixu}}{x}dx\right|\le 4\frac{\rho^{-1}}{|u|}
$$
По лемме Жордана,
$$
\lim\limits_{A\to\infty}\int\limits_{\rho < |x| < A}\frac{e^{\pm ix|u|}}{x}dx + i\int\limits_{\pm\pi}^0 e^{\pm i|u|\rho e^{i\varphi}}d\varphi = 0,
$$
значит,
$$
\left|\lim\limits_{A\to\infty}\int\limits_{\rho < |x| < A}\frac{e^{-ixu}}{x}dx\right|\le \pi, \
\left|\int\limits_{\rho < |x| < A}\frac{e^{-ixu}}{x}dx\right| \le 2\pi.
$$
\end{proof}

\begin{lemma}\label{lemma_double_int_estimate}
Пусть $f\in L_{s}(0, 2a)$, $1 < s \le \infty$, $0<t<2a$. Тогда при $\rho > \frac{3}{\pi a}$
\begin{equation}
\label{eq_double_int_estimate}
\left| \lim\limits_{A\to \infty}\int\limits_{\rho < |x| < A} \frac{e^{-ixt}}{x}\int\limits_{0}^{2a} f(y)e^{iyx}\,dy\,dx \right| \le 12 \Vert f\Vert_{s}\rho^{-\frac{s-1}{s}}\left(\ln\left(\frac{\pi a e}{3}\rho\right)\right)^{\frac{s-1}{s}}
.\end{equation}
\end{lemma}
\begin{proof}
Имеем
$$
\int\limits_{\rho < |x| < A}\frac{e^{-ixt}}{x}\int\limits_{0}^{2a}f(y)e^{ixy}\,dy\,dx =  \int\limits_0^{2a} f(y)\left[\int\limits_{\rho < |x| < A}\frac{e^{-ix(t-y)}}{x}\,dx \right]\,dy,
$$
поэтому $$
\left|\int\limits_{\rho < |x| < A}\frac{e^{-ixt}}{x}\int\limits_{0}^{2a}f(y)e^{ixy}\,dy\,dx\right|\le 2\pi\int\limits_{0}^{2a}|f(y)|\min\left( 1, \frac{3/\pi}{\rho|y-t|}\right) \,dy \le $$
$$\le 2\pi\Vert f\Vert_{s} \left( \int\limits_{0}^{2a}\min\left( 1, \frac{3/\pi}{\rho|y-t|}\right)^{\frac{s}{s-1}}\,dy\right)^{\frac{s-1}{s}}.
$$

Легко показать, что для всех $C \in (0, a]$ и $x\in [0, 2a]$ выполнено
\begin{equation}
\int\limits_0^{2a}\min\left( 1, \frac{C}{|x-y|}\right)^{\frac{s}{s-1}}\,dy \le \int\limits_0^{2a}\min\left( 1, \frac{C}{|x-y|}\right) \,dy\le \ 2C\ln\left(\frac{ae}{C}\right);
\end{equation}
взяв $C = \frac{3}{\pi \rho}$, где $\rho > \frac{3}{\pi a}$, получим, что
$$
 \int\limits_{0}^{2a}\min\left( 1, \frac{3/\pi}{\rho|y-t|}\right)^{\frac{s}{s-1}}\,dy\le \frac{6}{\pi}\rho^{-1}\log\left(\frac{\pi a e}{3}\rho\right),
$$
что и дает требуемое.
\end{proof}

Фиксируем параметр $\delta\in (0, 1)$, используемый в дальнейшем в лемме об оценке разности двух целых функций.
Разобьем интеграл \eqref{eq_Integral} на две части:
$$
\lim\limits_{A\to\infty}\int\limits_{-A}^{A}\left(\psi_2(x) - \psi_1(x)\right) e^{-ixt}dx = \left(\int\limits_{-R^{\alpha}}^{R^{\alpha}} + \int\limits_{|x| > R^{\alpha}}
\right) \left( \psi_2(x) - \psi_1(x)\right) e^{-ixt}\,dx,
$$
где $\alpha \in (0, 1 - \delta)$ --- произвольный параметр и используется обозначение
$$\lim\limits_{A\to \infty}\int\limits_{\rho < |x| < A}\varphi(x)\, dx = \int\limits_{|x| > \rho} \varphi(x) \, dx.$$

\bigskip

{\bf 5.1. Оценка интеграла $\int_{|x| > R^{\alpha}}(\psi_2(x) - \psi_1(x))e^{-ixt}dx$.}
\begin{lemma}
\label{lemma_K_1_K_2}

При $0 < t < 2a$ для всех $\alpha > 0$ и $R > \left(\frac{3}{\pi a}\right)^{\frac{1}{\alpha}}$ выполнено неравенство
\begin{equation}
\label{eq_integral_apriori}
\begin{gathered}
\left|\frac{1}{2\pi}\int\limits_{|x| > R^{\alpha}}\left(\psi_2(x) - \psi_1(x)\right) e^{-ixt}dx\right| \le  \frac{1}{2}Q_1e^{aQ_1}\min\left( 1, \frac{4/\pi}{R^{\alpha}t}\right) +\\
+ \frac{3}{\pi}\left( D_p + 4a^{\frac{1}{p}}Q_1^2e^{aQ_1}\right) R^{-\alpha\frac{p-1}{p}}\left(\ln\left(\frac{\pi a e}{3}R^{\alpha}\right)\right)^{\frac{p-1}{p}}.
\end{gathered}
\end{equation}

\end{lemma}

{\bf Доказательство.}
Поскольку
\begin{equation}
\label{eD_psi_represent1}
\psi_j(z) = 1 + \int\limits_{0}^{2a} K_{0j}(0, t)e^{izt}dt = 1 + \frac{i}{z}K_{0j}(0, 0) -\frac{i}{4z}\int\limits_{0}^{2a}g_j(t)e^{izt}dt,
\end{equation}
где
$$
g_j(t) = q_j\left(\frac{t}{2}\right)  - 4\frac{\partial H_{0j}}{\partial t}(0, t),
$$
то
$$\int\limits_{|x| > R^{\alpha}}\left(\psi_2(x) - \psi_1(x)\right) e^{-ixt}\,dx = \int\limits_{|x| > R^{\alpha}} e^{-ixt}\left( (K_{02} - K_{01})(0, 0)\frac{i}{x} - \frac{i}{4x}\int\limits_{0}^{2a} h(y)e^{ixy}\,dy\right) \,dx,
$$
где $h(y) = g_2(y) - g_1(y)$.

Из оценки \eqref{eq_int_sinc} с учетом $$
|K_{0j}(0, 0)|\le \frac{1}{2}Q_{1}e^{aQ_{1}},
$$ следует, что
при $t \in (0, 2a)$
$$\left|(K_{02}-K_{01})(0, 0)\int\limits_{|x| > R^{\alpha}}\frac{e^{-ixt}}{x}\,dx\right| \le \pi Q_{1}e^{aQ_1}\min \left( 1, \frac{4/\pi}{R^{\alpha}t}\right).
$$

Но из оценок \eqref{eq_KernelEstimates} следует $$
\Vert h\Vert_{p}\le 2^{\frac{1}{p}}\Vert q_2 - q_1\Vert_{p} + 4\left\Vert \frac{\partial H_{02}}{\partial t}(0, \cdot) - \frac{\partial H_{01}}{\partial t}(0, \cdot) \right\Vert_{p} \le 2^{\frac{1}{p}} D_p + 8a^{\frac{1}{p}}Q_1^2e^{aQ_1} \le 2D_p + 8a^{\frac{1}{p}}Q_1^2e^{aQ_1}.
$$
Значит, по Лемме \ref{lemma_double_int_estimate}, при $R > \left(\frac{3}{\pi a}\right)^{\frac{1}{\alpha}}$
\begin{equation}
\label{eq_intergral_estimate}
\begin{gathered}
\left|\int\limits_{R^{\alpha} < |x| < A}\frac{e^{-ixt}}{x}\int\limits_{0}^{2a}h(y)e^{ixy}\,dy\,dx\right|\le 24(D_p + 4a^{\frac{1}{p}}Q_1^2e^{aQ_1})R^{-\alpha\frac{p-1}{p}}\left(\ln\left(\frac{\pi a  e }{3}R^{\alpha}\right)\right)^{\frac{p-1}{p}}
\end{gathered}
\end{equation}

Подставляя полученные оценки в выражение для $\int\limits_{|x| > R^{\alpha}}(\psi_2(x)-\psi_1(x))e^{-ixt}dx$, получаем требуемое неравенство.

\bigskip

{\bf 5.2. Оценка интеграла $\int_{-R^{\alpha}}^{R^{\alpha}}(\psi_2(x) - \psi_1(x))e^{-ixt}dx$.}

\begin{lemma}
\label{lemma_MarchenkoEstimate}
Для $q\in L_1((0, \infty), xdx)$ при всех $z\in \overline{\mathbb{C}}_+$ имеют место оценки
\begin{equation}
\label{eq_MarchenkoEstimate}
|\psi_q(z) - 1|\le \frac{1}{|z|}\sigma_q(0)\exp\left(\overline{\sigma}_q(0)\right), \ |\psi_q(z)|\le \exp\left(\overline{\sigma}_q(0)\right).
\end{equation}
В частности, если $q\in L_1(0, a)$, то при $z\in \overline{\mathbb{C}}_+$
\begin{equation}
\label{eq_Marchenko_estimate}
|\psi_q(z)|\le e^{aQ_1}, \ \left|\psi_q(z) - 1\right|\le \frac{1}{|z|}Q_1e^{aQ_1},
\end{equation}
и, следовательно, $$Z(\psi)\cap \overline{\mathbb{C}}_+\subset \overline{B}(0, Q_1e^{aQ_1}).$$
\end{lemma}

{\bf Доказательство.}
Оценки \eqref{eq_MarchenkoEstimate} следуют из неравенств \cite[Лемма 3.1.3]{Marchenko}
$$
|y_q(x, z)|\le \exp\left(-\im(z)x + \overline{\sigma}_q(x)\right),
$$
$$
|y_q(x, z) - e^{izx}|\le \left(\overline{\sigma}_q(x) - \overline{\sigma}_q\left(x + \frac{1}{|z|}\right)\right)\exp\left(-\im(z)x + \overline{\sigma}_q(x)\right),
$$
верных для $x\in [0, \infty), z\in \overline{\mathbb{C}}_+$.

По теореме Адамара, поскольку порядок $\psi_j$ не превосходит 1,

$$
\psi_j(z) = z^{N(0)}e^{g_j(z)}\prod_{n=1}^{\infty}E\left(\frac{z}{z_n^{(j)}}\right),
$$
где
$$
E(w) = (1-w)e^{w}, \ g_j(z) = a_jz + b_j,
$$
а $\{z_n^{(j)}\}_{n=1}^{\infty} = Z(\psi_j)\setminus \{0\}$ --- последовательность отличных от нуля корней (с учетом кратности) функции $\psi_j$; по предположению, $Z(\psi_j)\cap\overline{B}(0, R)\setminus\{0\} = \{z_n^{(j)}\}_{n=1}^{N(R)}$.

Введем функции
$$\Pi_j(R, z) := \prod_{|z_n^{(j)}| >  R}E\left(\frac{z}{z_n^{(j)}}\right), \ W(z):= \prod\limits_{n=1}^{N(R)}\frac{E\left( z /z_n^{(2)}\right)}{E\left( z / z_n^{(1)}\right)}$$

Представим разность $\psi_2 - \psi_1$ в виде
$$
\psi_2(z) - \psi_1(z) = \psi_2(z)\left( 1 - \frac{1}{W(z)}e^{g_1(z) - g_2(z)}\frac{\Pi_1(R, z)}{\Pi_2(R, z)}\right) =
$$
$$
=\psi_2(z)\left(\left( 1 - e^{g_1(z) - g_2(z)}\frac{\Pi_1(R, z)}{\Pi_2(R, z)}\right) + e^{g_1(z) - g_2(z)}\frac{\Pi_1(R, z)}{\Pi_2(R, z)}\left( 1 - \frac{1}{W(z)}\right)\right) =
$$
$$
= \psi_1(z)\left( W(z) - 1\right) + \psi_2(z)\left( 1 - e^{g_1(z) - g_2(z)}\frac{\Pi_1(R, z)}{\Pi_2(R, z)}\right).
$$
Согласно лемме \ref{lemma_MarchenkoEstimate}, при $z\in \overline{\mathbb{C}}_+$ верны оценки
$$
|\psi_j(z)|\le e^{aQ_1}, \ |\psi_2(z) - \psi_1(z)|\le \frac{2Q_1e^{aQ_1}}{|z|},
$$ поэтому для оценки интеграла
$\int\limits_{-R^{\alpha}}^{R^{\alpha}}(\psi_2(x) - \psi_1(x))e^{-ixt}dx$
достаточно оценить функции
$1 - e^{g_1(z)-g_2(z)}\Pi_1(R, z)/\Pi_2(R, z)$ и $W(z) - 1$ на пути интегрирования, в качестве которого (вместо отрезка $[-R^{\alpha}, R^{\alpha}]$) возьмем контур
$$
\Gamma(A, R^{\alpha}) = \left[-R^{\alpha}, -R^{\alpha} + iA\right]\cup \left[-R^{\alpha} + Ai, R^{\alpha} + Ai\right] \cup \left[R^{\alpha} + Ai, R^{\alpha}\right], \ A:= 1 + Q_1e^{aQ_1}.
$$
Для всех $z\in \left[-R^{\alpha} + Ai, R^{\alpha} + Ai\right]$ выполнено неравенство
\begin{equation}
\label{eq_on_contour}
|z - z_n^{(j)}| \ge 1,
\end{equation}
которое будет использоваться при оценке $|W(z) - 1|$.

\begin{lemma}
\label{lemma_main_analytic}
 Для любых $\delta\in (0, 1), \eta > 0$ существуют такое $R_1 > 0$ зависящее от $\delta$, $Q_1$, $a$ и $\eta$, что для $z\in B(0, R^{1-\delta})$ выполнено неравенство
 \begin{equation}
 \left| 1 - e^{g_1(z) - g_2(z)}\frac{\Pi_1(R, z)}{\Pi_2(R, z)} \right|\le \eta R^{-(1-\delta)},
 \end{equation}
  при условии, что $R > R_1$.

Если дополнительно потребовать выполнение равенства
$$
\int\limits_{0}^aq_1(s)ds = \int\limits_0^aq_2(s)ds,
$$
то найдется такое $R_2 = R_2(\delta, Q_1, a, \eta)$, что для $z\in B(0, R^{1-\delta})$
$$
\left| 1- e^{g_1(z) - g_2(z)}\frac{\Pi_1(R, z)}{\Pi_2(R, z)} \right|\le \eta R^{-2(1-\delta)},
$$
если $R > R_2$.
\end{lemma}

{\bf Доказательство.}
Для потенциала $q\in L_1(0, a)$ такого, что $\Vert q\Vert_1\le Q_1$, при всех $z\in \mathbb{C}$ имеем
$$
\left|\psi_q(z) - 1\right| = \left| \int\limits_{0}^{2a}K_q(0, t)e^{izt}dt\right|\le \frac{1}{2}Q_1e^{aQ_1}\int\limits_{0}^{2a}e^{-\im(z)t}dt, $$$$|\psi_q(z)|\le 1 + aQ_1e^{aQ_1}e^{2a|z|}\le (1 + aQ_1e^{aQ_1})e^{2a|z|},
$$ а при $x > 0$
$$
|\psi_q(ix) - 1|\le \frac{1}{2}Q_1e^{aQ_1}\frac{1}{x}.
$$
Значит, класс функций Йоста для $\Vert q\Vert_1\le Q_1$ удовлетворяет условиям Теоремы 1 \cite{GeyntsShkalikov} со следующими параметрами: $r_0 > 0$ любое, $C_0 = 1 + aQ_1e^{aQ_1}$, $C_1 = \frac{1}{2}Q_1e^{aQ_1}$, $\rho = 1$, $\sigma = 2a$, $\varphi = \frac{\pi}{2}$; отсюда следует первое утверждение леммы. Второе утверждение следует из Замечания 6 работы \cite{GeyntsShkalikov}.

\begin{lemma}
\label{lemma_N_estimate}
Пусть $q\in L_1(0, a)$ и $\Vert q\Vert_{1}\le Q_1$. Тогда при
$$R \ge R_3(a, Q_1) := \frac{1}{ae}\left(aQ_1(1 + 2e^{aQ_1}(e + 1)) + \ln(2(aQ_1 + 1))\right)$$ выполнено неравенство
\begin{equation}
N_{\psi_q}(R)\le 3aeR.
\end{equation}
\end{lemma}


Доказательство этой леммы основано на применении формулы Йенсена и аналогично доказательству утверждения работы \cite{GeyntsShkalikov}.

\begin{lemma}
\label{lemma_K_2_minus_K_1_estimate}
Для каждого $\alpha\in (0, 1 - \delta)$ существует такое $R_4 = R_4(\delta, Q_1, a, \alpha)$, что для всех $t\in (0, 2a)$ при условии $R > R_4$  выполнено
\begin{equation}
\label{eq_K_2_minus_K_1_estimate}
|K_{02}(0, t) - K_{01}(0, t)|\le \frac{1}{2}Q_1e^{aQ_1}\min\left( 1, \frac{4}{\pi R^{\alpha}t}\right) + \widehat{\Phi}
\end{equation}
где
$$\widehat{\Phi} = \Phi + \Phi_{\varepsilon},$$
$$\Phi:=\frac{3}{\pi}\left( D_p + 4a^{\frac{1}{p}}Q_1^2e^{aQ_1}\right) R^{-\alpha\frac{p-1}{p}}\left(\ln\left(\frac{\pi a e}{3}R^{\alpha}\right)\right)^{\frac{p-1}{p}} + R^{\alpha - 1 + \delta} + $$
$$ + \frac{2(1 + Q_1e^{aQ_1})Q_1}{\pi}e^{a(Q_1 + 2(1 + Q_1e^{aQ_1}))}R^{-\alpha},
$$
$$\Phi_{\varepsilon} = \frac{2}{\pi}e^{2a(1 + Q_1e^{aQ_1})}R^{\alpha}\varphi_{\alpha}(R, \varepsilon),
$$
где $\varphi_{\alpha}$ определено формулой \eqref{eq_varphi}.
\end{lemma}
{\bf Доказательство.}
Оценим
$$
\frac{1}{2\pi}\int\limits_{\Gamma(A, R^{\alpha})}\left(\psi_2 - \psi_1\right)(z) e^{-izt}dz = \frac{1}{2\pi}\left(\int\limits_{-R^{\alpha} + Ai}^{R^{\alpha} + Ai} + \int\limits_{-R^{\alpha}}^{-R^{\alpha} + Ai} + \int\limits_{R^{\alpha} + Ai}^ {R^{\alpha}}\right) \left(\psi_2 - \psi_1\right)(z) e^{-izt}dz.
$$
Имеем
$$
\frac{1}{2\pi}\left|\int\limits_{\pm R^{\alpha}}^{\pm R^{\alpha} + Ai}(\psi_2(z) - \psi_1(z))e^{-izt}dz\right|\le \frac{1}{2\pi}A\frac{2Q_1e^{aQ_1}}{R^{\alpha}}e^{2aA} = \frac{AQ_1e^{aQ_1 + 2aA}}{\pi R^{\alpha}}.
$$
Далее, взяв $\eta = \frac{\pi}{2}e^{-2aA}$ в лемме \ref{lemma_main_analytic}, получим при $R > \max\left\{ R_1 = R_1(\delta, Q_1, a, \eta), \exp\left(\frac{\ln(2)}{1-\delta - \alpha}\right)\right\{$
$$
\frac{1}{2\pi}\left|\int\limits_{-R^{\alpha} + Ai}^{R^{\alpha} + Ai}(\psi_2 - \psi_1)(z)e^{-izt}dz\right|\le \frac{1}{2\pi}\int\limits_{-R^{\alpha} + Ai}^{R^{\alpha} +Ai}|\psi_2(z)|\left|1 - e^{g_1(z) - g_2(z)}\frac{\Pi_1(R, z)}{\Pi_2(R, z)}\right| e^{\im(z)t}dz +
$$
$$
+\frac{1}{2\pi}\int\limits_{-R^{\alpha} + Ai}^{R^{\alpha} + Ai}|\psi_1(z)||W(z) - 1|e^{\im(z)t}dz \le R^{\alpha - 1 + \delta} + \frac{1}{\pi}e^{2aA}\int\limits_{-R^{\alpha} + Ai}^{R^{\alpha} + Ai}|W(z)-1|dz.
$$
Требуется оценить при $z\in [-R^{\alpha} + Ai, R^{\alpha} + Ai]$ величину
$
W(z) - 1 = U(z)e^{S(z)} - 1$, где
$$
U(z) = \prod\limits_{n=1}^{N(R)}(1 + u_n), \ u_n = zz_n^{(1)}\left(\frac{1}{z_n^{(1)}}- \frac{1}{z_n^{(2)}}\right)\frac{1}{z_n^{(1)} - z},$$
$$
S(z) = z\sum\limits_{n = 1}^{N(R)}\left(\frac{1}{z_n^{(2)}} - \frac{1}{z_n^{(1)}}\right).
$$
Имеем
$$
|U(z)e^{S(z)} - 1|\le |U(z) - 1|e^{|S(z)|} + |e^{S(z)} - 1|\le |U(z) - 1|e^{|S(z)|} + |S(z)|e^{|S(z)|}.
$$
Т.к. при $R^{\alpha} > A$ для $z\in [-R^{\alpha} + Ai, R^{\alpha} + Ai]$ выполнено $|z|\le 2R^{\alpha}$ и $|z_n^{(j)} - z|\ge 1$, то
$
|u_n|\le 2R^{1+\alpha}\varepsilon,
$
поэтому при $R > R_3$
$$
|U(z) - 1|\le \sum\limits_{n=1}^{N(R)}\binom{N(R)}{n}(2R^{1 + \alpha}\varepsilon)^n \le \sum\limits_{n=1}^{\lceil 3aeR\rceil}\binom{\lceil 3aeR\rceil}{n}(2R^{1 + \alpha}\varepsilon)^n,
$$
и
$$
|S(z)|\le 2N(R)\varepsilon R^{\alpha}\le 6aeR^{1 + \alpha}\varepsilon,
$$
так что
$$
|W(z) - 1|\le \varphi_{\alpha}(R, \varepsilon), $$
где
\begin{equation}
\label{eq_varphi}
\varphi_{\alpha}(R, \varepsilon):=
\sum\limits_{n=1}^{\lceil 3aeR\rceil}\binom{\lceil 3aeR\rceil}{n}(2R^{1 + \alpha}\varepsilon)^n\exp\left( 6aeR^{1+\alpha}\varepsilon\right) + 6aeR^{1 + \alpha}\varepsilon\exp\left( 6aeR^{1+\alpha}\varepsilon\right).
\end{equation}

Наконец,
$$
\left|\frac{1}{2\pi}\int\limits_{-R^{\alpha}}^{R^{\alpha}}(\psi_2(x) - \psi_1(x))e^{-ixt}dx\right| = \frac{1}{2\pi}\left| \int\limits_{\Gamma(A, R^{\alpha})}(\psi_2(z) - \psi_1(z))e^{-izt}dz\right|\le
$$
$$
\le \frac{2AQ_1}{\pi}e^{aQ_1 + 2aA}R^{-\alpha} + R^{\alpha - 1 + \delta} + \frac{2}{\pi}e^{2aA}R^{\alpha}\varphi_{\alpha}( R, \varepsilon).
$$

Складывая это неравенство с  \eqref{eq_integral_apriori} при $$R > R_4 := \max\left\{ R_1, R_3, (1 + Q_1e^{aQ_1})^{\frac{1}{\alpha}}, \left(\frac{3}{\pi a}\right)^{\frac{1}{\alpha}}, \exp\left(\frac{\ln(2)}{1-\delta - \alpha}\right)\right\},$$ получаем требуемое.


\bigskip

{\bf 6. Оценка ядра $K_{12}(x, t)$.}
\begin{lemma}
\label{lemma_K_12_estimate}
В обозначениях леммы \ref{lemma_K_2_minus_K_1_estimate} при условии $R > R_4$ для $t\in (0, 2a)$ выполнено
\begin{equation}
\label{eq_K_12_estimate}
|K_{12}(0, t)|\le \frac{1}{2}Q_1e^{aQ_1}\min\left( 1, \frac{4}{\pi R^{\alpha}t}\right) + \widehat{\Psi},\end{equation} где
$$
\widehat{\Psi} = \Psi + \Psi_{\varepsilon},
$$
$$
\Psi := \frac{Q_1^2e^{2aQ_1}}{\pi R^{\alpha}}\ln\left(\frac{\pi a e R^{\alpha}}{2}\right) +(1 + aQ_1e^{aQ_1})\Phi,$$
$$
\Psi_{\varepsilon} = (1 + aQ_1e^{aQ_1})\Phi_{\varepsilon}.
$$
\end{lemma}
{\bf Доказательство}

Подставляем оценку леммы \ref{lemma_K_2_minus_K_1_estimate} в равенство \eqref{eq_K12_K1_K_2}, учитывая $|K_{j0}(x, t)|\le \frac{1}{2}Q_1e^{aQ_1}$.

\begin{lemma}\label{lemma_Kernel_estimate}
Для всех $\alpha\in (0, 1-\delta)$ при условии $R > R_4$ выполнено неравенство
\begin{equation}\label{eq_Kernel_estimate}|K_{12}(x, t)|\le \frac{1}{2}Q_1e^{aQ_1}\min\left( 1, \frac{4}{\pi R^{\alpha}(x+t)}\right) + \widehat{\Omega},
\end{equation}
 где
 $$
\widehat{\Omega}  = \Omega + \Omega_{\varepsilon},
 $$
 $$\Omega = e^{2aQ_1}\left( \Psi + \frac{2Q_1^2e^{aQ_1}}{\pi R^{\alpha}}\ln\left(\frac{\pi a e  R^{\alpha}}{2}\right)\right), \ \Omega_{\varepsilon} = e^{2aQ_1}\Psi_{\varepsilon}.$$

\end{lemma}
{\bf Доказательство.} Перейдем к переменным $(u, v)$:
$$K_{12}(x, t) = \widetilde{K}_{12}\left(\frac{t-x}{2}, \frac{t + x}{2}\right),$$
где $\widetilde{K}_{12}$ --- решение интегрального уравнения \ref{eq_K12_integral_equation}, которое находится методом итераций: $$\widetilde{K}_{12}(u, v) = \sum\limits_{n=0}^{\infty}\widetilde{K}_{12}^{(n)}(u, v), \
\widetilde{K}_{12}^{(0)}(u, v) =\widetilde{K}_{12}(v, v) = K_{12}(0, 2v),$$
$$  \ \widetilde{K}^{(n+1)}_{12}(u, v) = \int\limits_{u}^vdy\int\limits_{v}^{a}dz\left( q_1(z + y) - q_2(z - y)\right)\widetilde{K}_{12}^{(n)}(y, z).
$$
По лемме \ref{lemma_K_12_estimate}, при $R > R_4$
$$
|\widetilde{K}_{12}^{(0)}(u, v)| \le \frac{1}{2}Q_1e^{aQ_1}\min\left( 1, \frac{2}{\pi R^{\alpha}v}\right) + \widehat{\Psi},
$$
$$
\left|\widetilde{K}_{12}^{(1)}(u, v)\right| \le 2Q_1(a-v)\widehat{\Psi} + \frac{2Q_1^2e^{aQ_1}}{\pi R^{\alpha}}\ln\left(\frac{\pi a e}{2}R^{\alpha}\right)
$$
Далее, для всех $n\ge 0$
$$
\left| \widetilde{K}_{12}^{(n+1)}(u, v)\right| \le \frac{(2Q_1(a-v))^{n+1}}{(n+1)!}\widehat{\Psi} + \frac{2Q_1^2e^{aQ_1}}{\pi R^{\alpha}}\ln\left( \frac{\pi a e}{2}R^{\alpha}\right)\frac{(2Q_1(a-v))^{n}}{n!} ,
$$
$$
|\widetilde{K}_{12}(u, v)|\le \sum\limits_{n=0}^{\infty}|\widetilde{K}_{12}^{(n)}(u, v)|\le
$$
$$\le\frac{1}{2}Q_1e^{aQ_1}\min\left( 1, \frac{2}{\pi R^{\alpha}v}\right) + \widehat{\Psi} e^{2Q_1(a-v)} + \frac{2Q_1^2e^{aQ_1}}{\pi R^{\alpha}}\ln\left( \frac{\pi a e R^{\alpha}}{2}\right) e^{2Q_1(a-v)},
$$
откуда  и следует требуемое.

\begin{theorem}
\label{th_MainTheorem}
Существует такая константа $R_0 = R_0(\delta, Q_1, D_p, a)$, что при условии $R > R_0$ выполнена оценка
\begin{equation}\label{eq_MainEstimate}
\max\limits_{0\le x\le a}\left| \int\limits_{x}^{a}(q_2(t) - q_1(t))\,dt\right| \le \psi(R, \varepsilon) + C_0R^{-(1-\delta)\frac{p-1}{3p-2}}(1 + \chi(R)),
\end{equation}
где $$C_0 = 2\left(D_p^{\frac{p}{2p-1}}\left(\frac{2Q_1e^{aQ_1}}{\pi}\right)^{\frac{p-1}{2p-1}} + e^{2aQ_1}(1 + aQ_1e^{aQ_1})\right),
$$
а $\chi(R) = O\left(R^{-(1-\delta)\frac{(p-1)^2}{p(3p-2)}}(\ln R)^{\frac{p-1}{p}}\right)$ при $R\to\infty$,
$\psi(R, \varepsilon) = O(\varepsilon)$ при $\varepsilon\to 0$.
\end{theorem}
{\bf Доказательство.}
 По лемме \ref{lemma_Kernel_estimate}, при условии $R > R_4$ имеем для всех $x \in (0, a]$:
$$
\left|\int\limits_{x}^{a}(q_2(t) - q_1(t))dt\right| = 2\left| K_{12}(x, x)\right|\le Q_1e^{aQ_1}\min\left( 1, \frac{2}{\pi R^{\alpha}x}\right) + 2\widehat{\Omega}.
$$

Обозначив $D:=\frac{2Q_1e^{aQ_1}}{\pi }R^{-\alpha}$, для фиксированного $\theta\in\left( \frac{2}{\pi}R^{-\alpha}, a\right)$ и для $0\le x\le \theta$ имеем
$$
\left|\int\limits_{x}^a(q_2(t) - q_1(t))dt\right|\le \int\limits_{0}^{\theta}|q_2(t) - q_1(t)|dt + \left|\int\limits_{\theta}^{a}(q_2(t) - q_1(t))dt\right| \le f(\theta), $$
где
$$f(\theta):=D_p\theta^{\frac{p-1}{p}} + \frac{D}{\theta} + 2\widehat{\Omega};
$$
для $x\in(\theta, a]$ это неравенство выполнено тривиальным образом.

Выберем $R_0$ настолько большим, чтобы при $R > R_0$ точка минимума
$$
\theta_0 = \left(\frac{p}{p-1}\right)^{\frac{p}{2p-1}}\left(\frac{D}{D_p}\right)^{\frac{p}{2p-1}}
$$ функции $f(\theta)$ принадлежала интервалу $(\frac{2}{\pi}R^{-\alpha}, a)$: пусть $$R_0 := \max\left( R_4, \left(\frac{2Q_1 e^{aQ_1}}{\pi D_p}\right)^{\frac{1}{\alpha}}\left(\frac{p}{p-1}\right)^{\frac{1}{\alpha}}a^{-\frac{2p-1}{p\alpha}}, \left(\frac{p-1}{p}\right)^{\frac{p}{\alpha(p-1)}}\left(\frac{2}{\pi}\right)^{\frac{2p-1}{(p-1)\alpha}}\left(\frac{\pi D_p}{2Q_1e^{aQ_1}}\right)^{\frac{p}{\alpha(p-1)}}\right).$$
Тогда
$$
f(\theta_0) = A_p D_p^{\frac{p}{2p-1}} D^{\frac{p-1}{2p-1}} + 2\widehat{\Omega}
$$
---минимальное значение $f(\theta)$ на отрезке $[\frac{2}{\pi}R^{-\alpha}, a]$,
где
$$A_p=\left(\left(\frac{p-1}{p}\right)^{\frac{p}{2p-1}} + \left(\frac{p}{p-1}\right)^{\frac{p-1}{2p-1}}\right)\le 2.$$

Итак,
$$
2|K_{12}(x, x)|\le
2D_p^{\frac{p}{2p-1}}\left(\frac{2Q_1e^{aQ_1}}{\pi}\right)^{\frac{p-1}{2p-1}} R^{-\alpha\frac{p-1}{2p-1}} + 2\widehat{\Omega} = 2D_p^{\frac{p}{2p-1}}\left(\frac{2Q_1e^{aQ_1}}{\pi}\right)^{\frac{p-1}{2p-1}} R^{-\alpha\frac{p-1}{2p-1}}  +
$$
$$
+2e^{2aQ_1}\left((1 + aQ_1e^{aQ_1})\frac{3}{\pi}(D_p + 4a^{\frac{1}{p}}Q_1^2e^{aQ_1})R^{-\frac{p-1}{p}\alpha}\ln\left(\frac{\pi a e}{3}R^{\alpha}\right)^{\frac{p-1}{p}} + (1 + aQ_1e^{aQ_1})R^{\alpha - 1 + \delta}\right) +
$$
$$
+ e^{3aQ_1}\frac{2Q_1^2}{\pi}(2 + e^{aQ_1})R^{-\alpha}\ln\left(\frac{\pi a e}{2}R^{\alpha}\right) + $$$$ +e^{3aQ_1 + 2a(1 + Q_1e^{aQ_1})}\left(\frac{4Q_1(1 + aQ_1e^{aQ_1})(1 + Q_1e^{aQ_1})}{\pi}\right) R^{-\alpha} +
$$
$$
+4e^{2aQ_1}\frac{1 + aQ_1e^{aQ_1}}{\pi}e^{2a(1 + Q_1e^{aQ_1})}R^{\alpha}\varphi_{\alpha}(R, \varepsilon).
$$

Выберем $\alpha\in[0, 1 - \delta]$ так, чтобы минимизировать максимум из показателей степеней $R$ входящих в эту формулу, кроме последнего слагаемого с $\varepsilon$:
$$
\alpha_* = arg \min\limits_{\alpha\in [0, 1 - \delta]}\max\left\{-\alpha\frac{p-1}{2p-1}, \ \alpha - 1 + \delta\right\} = (1-\delta)\frac{2p-1}{3p-2}.
$$

Тогда
$$
2|K_{12}(x, x)|\le C_0(1 + \chi(R))R^{-(1-\delta)\frac{p-1}{3p-2}} + \psi_{\alpha_*}(R, \varepsilon), $$
где $$C_0 =\left(D_p^{\frac{p}{2p-1}}\left(\frac{2Q_1e^{aQ_1}}{\pi}\right)^{\frac{p-1}{2p-1}} + e^{2aQ_1}(1 + aQ_1e^{aQ_1})\right),
$$
\begin{equation*}
\begin{split}
&C_0\chi(R) = \frac{6}{\pi}e^{2aQ_1}(1 + aQ_1e^{aQ_1})(D_p + 4a^{\frac{1}{p}}Q_1^2 e^{aQ_1})R^{-(1-\delta)\frac{(p-1)^2}{p(3p-2)}}\ln\left(\frac{\pi a e}{3}R^{(1-\delta)\frac{2p-1}{3p-2}}\right)^{\frac{p-1}{p}}+\\
&+ \frac{4}{\pi}e^{3aQ_1 + 2a(1 + Q_1e^{aQ_1})}Q_1(1 + aQ_1e^{aQ_1})(1 + Q_1e^{aQ_1})R^{-(1-\delta)\frac{p}{3p-2}} + \\
&+\frac{2}{\pi}Q_1^2e^{3aQ_1}(2 + e^{aQ_1})\ln\left(\frac{\pi a e}{2}R^{
(1-\delta)\frac{2p-1}{3p-2}}\right)R^{-(1-\delta)\frac{p}{3p-2}} = O\left(R^{-(1-\delta)\frac{(p-1)^2}{p(3p-2)}}(\ln R)^{\frac{p-1}{p}}\right)\end{split}
\end{equation*}
при $R\to \infty$ и
\begin{equation*}
\psi_{\alpha}(R, \varepsilon) = 4e^{2aQ_1}\frac{1 + aQ_1e^{aQ_1}}{\pi}e^{2a(1 + Q_1e^{aQ_1})}R^{\alpha}\varphi_{\alpha}(R, \varepsilon) = O(\varepsilon), \ \varepsilon\to 0.
\end{equation*}

\noindent

{
{\bf 7. Оценка $|q_2(x) - q_1(x)|$ при дополнительных априорных условиях.}

План получения оценки $|q_2(x) - q_1(x)|$ состоит в следующем. При наложении условия гладкости на разность $q_2 - q_1$ и выполнения равенства $\int\limits_{0}^a(q_2(t) - q_1(t))dt = 0$ можно получить оценку
$$
|\psi_2(z) - \psi_1(z)|\le \frac{C}{|z|^2}, \ z\in \overline{\mathbb{C}}_+;
$$
это дает возможность получить оценку величины $\frac{d}{dt}\left(K_{02}(0,t) - K_{01}(0, t)\right)$, аналогичную лемме \ref{lemma_K_2_minus_K_1_estimate}.
Далее, оцениваем величину
\begin{equation}
\label{eq_pd_K12_estimate}
\begin{split}
&\frac{\partial K_{12}}{\partial t}(0, t) = \frac{d}{dt}\left((K_{02} - K_{01})(0, t)\right) + (K_{02} - K_{01})(0, t)K_{10}(t, t) + \\
& + \int\limits_{0}^t(K_{02} - K_{01})(0, y)\frac{d}{dy}\left(-\frac{1}{2}\int\limits_{\frac{y+t}{2}}^a q_1(u)du\right) \,dy + \int\limits_{0}^t(K_{02} - K_{01})(0, y)\frac{\partial H_{10}}{\partial t}(y, t)\,dy,
\end{split}
\end{equation}
где
$$
K_{10}(x, t) = -\frac{1}{2}\int\limits_{\frac{t+x}{2}}^a q_1(y)\,dy + H_{10}(x, t),
$$
и $\frac{\partial H_{10}}{\partial t}$ допускает оценку $|\frac{\partial H_{10}}{\partial t}(x, t)|\le \frac{Q_1^2}{2}e^{aQ_1}$.
Наконец, равенства
$$
q_2(x) - q_1(x) = -2\frac{d}{dx}K_{12}(x, x) = -2\frac{d}{dx}\widetilde{K}_{12}(0, x), \ x\in (0, a), $$
$$\frac{d}{dx}\widetilde{K}_{12}(0, x) = \int\limits_{x}^a  (q_1(y + x) - q_2(y - x))\widetilde{K}_{12}(x, y)\,dy -$$
$$-\int\limits_{0}^x  (q_1(x + y) - q_2(x - y))\widetilde{K}_{12}(y, x)\,dy + 2\frac{\partial K_{12}}{\partial t}(0, 2x),  \ x\in (0, a),
$$
позволяют получить оценку $|q_2(x) - q_1(x)|$ при $x\in (0, a)$.

Положим
\begin{equation}
\label{eq_norm_embedding}
A_{\infty}:=\sup\left\{ \Vert q\Vert_{\infty}: \ q\in AC[0, a], \ \Vert q\Vert_{p} + \Vert q'\Vert_{r}\le 1\right\}.
\end{equation}

\begin{theorem}
\label{th_SecondTheorem}
Предположим, что в дополнение к условиям, сформулированным в пункте 3, выполнены следующие условия:
\begin{equation}
\label{eq_add_conditions}
\int\limits_0^a{q_1}(t)dt = \int\limits_0^a q_2(t)dt, \ q_2 - q_1 \in AC[0, a], \ \Vert(q_2 - q_1)'\Vert_r\le D_r'
\end{equation}
для некоторого $r\in (1, \infty]$.
Тогда существуют константы $C_0'$ и $R_0'$ такие, что при условии $R > R_0'$
выполнена оценка
\begin{equation}
\label{eq_norm_estimate}
\Vert q_1 - q_2\Vert\le \psi'(R, \varepsilon) + C_0' R^{-(1-\delta)\frac{r-1}{2r-1}\frac{p-1}{3p-2}}(1 + \chi'(R)),
\end{equation}
где $\chi'(R) = O(\ln(R))$ при $R\to \infty$, $\psi'(R, \varepsilon) = O(\varepsilon)$ при $\varepsilon\to 0$.

\end{theorem}

Заметим, что условие на разность потенциалов $q_2-q_1$ влечет
$$
\Vert q_2 - q_1\Vert_{\infty}\le A_{\infty}(D_p + D_{r}') =: D.
$$

Доказательство этой теоремы включает в себя несколько лемм, в которых предполагаются выполненными условия \eqref{eq_add_conditions}.
\begin{lemma}
\label{lemma_pd_2H_estimate}
Функция $[0, 2a]\ni t\mapsto \frac{\partial H_{0j}}{\partial t}(0, t)$ абсолютно непрерывна, и почти для всех $t\in (0, 2a)$ справедливы оценки
$$
\left|\left(\frac{\partial^2 H_{02}}{\partial t^2} - \frac{\partial^2 H_{01}}{\partial t^2}\right)(0, t) + \frac{1}{4}\left[\rho_2q_2-\rho_1q_1\right]\left(\frac{t}{2}\right) \right|\le Q_1^3e^{aQ_1} + \frac{3}{8}Q_1D,
$$
и
$$
\left|\frac{1}{4}\left[\rho_2q_2 - \rho_1q_1\right]\left(\frac{t}{2}\right)\right|\le \frac{1}{4}Q_1D +\frac{1}{2}Q_1\left| q_2\left(\frac{t}{2}\right)\right|.
$$
\end{lemma}

{\bf Доказательство.}
Введем следующие обозначения:
\begin{equation}
\begin{split}
&h_j(u, v): = q_j(u)K_{0j}(u, v), \ u\le v,\\
&\rho_j(x): = \int\limits_{x}^a q_j(t)dt;
\end{split}
\end{equation}
тогда $$\rho_1(0) = \rho_2(0), \ |\rho_j(x)|\le Q_1, \ |h_j(u, v)|\le \frac{Q_1}{2}e^{aQ_1}|q_j(u)|, \ j = 1, 2.$$
Легко показать, что
$$
\frac{\partial H_{0j}}{\partial t}(x, t) = \frac{1}{4}\rho_j^2\left(\frac{t+x}{2}\right) - \frac{1}{4}\int\limits_{\frac{t +x}{2}}^a \rho_j\left( v - \frac{t -x}{2}\right) q_j(v)dv + $$
$$+ \frac{1}{2}\int\limits_{0}^{\frac{t -x}{2}}du\int\limits_{\frac{t + x}{2}}^a dv \left[\rho_j\left(\frac{t + x}{2} - u\right) - \rho_j\left( v - \frac{t - x}{2}\right)\right]q_j(v - u)K_{0j}(v - u, v+u);
$$
отсюда следует абсолютная непрерывность функции $\frac{\partial  H_{0j}}{\partial  t}(0, \cdot)$.
Дифференцированием получаем
$$
\frac{\partial^2 H_{0j}}{\partial t^2}(0, t) = \sum\limits_{k=1}^6 I_j^{(k)},
$$
где
$$
I_j^{(1)} = -\frac{1}{4}\int\limits_0^{\frac{t}{2}}du\int\limits_{\frac{t}{2}}^a dv \left( q_j\left( \frac{t}{2} - u\right)  + q_j\left( v - \frac{t}{2}\right)\right) h_j(v-u, v+u).
$$
$$
I_j^{(2)} = - \frac{1}{4}\int\limits_0^{\frac{t}{2}}\left(\rho_j\left(\frac{t}{2} - u\right) - \rho_j(0)\right) h_j\left(\frac{t}{2} - u, \frac{t}{2} + u\right)
du,
$$
$$
I_j^{(3)} = \frac{1}{4}\int\limits_{\frac{t}{2}}^a \left( \rho_j(0) - \rho_j\left( v - \frac{t}{2}\right)\right) h_j\left( v - \frac{t}{2}, v + \frac{t}{2}\right) dv,
$$
$$
I_j^{(4)} = - \frac{1}{8}\int\limits_{\frac{t}{2}}^a  q_j\left( v-\frac{t}{2}\right) q_j(v) dv, \ I_j^{(5)} = \frac{1}{8}\rho_j(0)q_j\left(\frac{t}{2}\right), \ I_j^{(6)} = -\frac{1}{4}\rho_j\left(\frac{t}{2}\right) q_j\left(\frac{t}{2}\right).
$$
Имеем
$$
|I_2^{(1)} - I_1^{(1)}|\le \frac{1}{2}Q_1^3e^{aQ_1},  \ |I_2^{(2)} - I_1^{(2)}|\le \frac{1}{4}Q_1^3e^{aQ_1},  \ |I_2^{(3)} - I_1^{(3)}|\le \frac{1}{4}Q_1^3e^{aQ_1},
$$
$$
|I_2^{(4)} - I_1^{(4)}|\le \frac{1}{8}\int\limits_{\frac{t}{2}}^a dv\left( |q_2(v - t/2) - q_1(v-t/2)||q_2(v)| + |q_1(v-t/2)||q_2(v) - q_1(v)|\right)\le
$$
$$\le \frac{1}{8}DQ_1 + \frac{1}{8}DQ_1 = \frac{1}{4}DQ_1,$$
$$|I_2^{(5)} - I_1^{(5)}|\le \frac{1}{8}Q_1D,$$

Складывая предыдущие оценки, получаем требуемое неравенство.

\begin{lemma}
При $z\in \overline{\mathbb{C}}_+$ справедлива оценка
\begin{equation}
\label{eq_strong_estimate}
|\psi_2(z) - \psi_1(z)|\le \frac{F}{|z|^2},
\end{equation}
где
\begin{equation}
\label{eq_F}
F:=\frac{1}{4}\left( 2D  +8Q_1^2e^{aQ_1} + a^{\frac{r-1}{r}}D_r'+ 4Q_1^2  + 8a\left( Q_1^3e^{aQ_1} + \frac{5}{8}Q_1D\right) \right).
\end{equation}

\end{lemma}

{\bf Доказательство.}
Поскольку $K_{0j}(0, 0) = \frac{1}{2}\int\limits_{0}^a q_j(t)dt$, и по условию, $K_1(0, 0) = K_2(0, 0)$, то из \eqref{eD_psi_represent1} получаем для всех $z\neq 0$:
$$
\psi_2(z) - \psi_1(z) = \frac{i}{z}\int\limits_0^{2a}\left(\frac{\partial K_{02}}{\partial t} - \frac{\partial K_{01}}{\partial t}\right)(0, t)e^{izt}dt = -\frac{i}{4z}\int\limits_0^{2a}h(y)e^{izy}ds,
$$
$$
h(y) = q_2\left(\frac{y}{2}\right) - q_1\left(\frac{y}{2}\right) - 4\left(\frac{\partial H_{02}}{\partial t}- \frac{\partial H_{01}}{\partial t}\right)(0, y),  \ y\in(0, 2a).
$$
Получаем оценку
$$
\Vert h\Vert_{\infty}\le D + 4Q_1^2e^{aQ_1},
$$
а из $q_2 - q_1 \in AC[0, a]$ следует, что $h\in AC[0, 2a]$, поэтому
$$
\int\limits_{0}^{2a}h(y)e^{izy}\,dy =\frac{1}{iz}\left( h(2a)e^{2aiz} - h(0)\right) - \frac{1}{iz}\int\limits_{0}^{2a}h'(y)e^{izy}d\,y
$$  и
\begin{equation}
\label{eq_strong_representation}
\begin{split}
& \psi_2(z) - \psi_1(z) = -\frac{1}{4z^2}\left( h(2a)e^{2aiz} - h(0)\right) + \\
& + \frac{1}{8z^2}\int\limits_{0}^{2a} e^{izy}\left( q_2 - q_1\right)'\left(\frac{y}{2}\right)\,dy -\frac{1}{z^2}\int\limits_0^{2a} e^{izy}\left(\frac{\partial^2 H_{02}}{\partial t^2} - \frac{\partial^2 H_{01}}{\partial t^2}\right)(0, y)\,dy.
\end{split}
\end{equation}

Тогда
$$
|h(2a)e^{2aiz} - h(0)|\le 2(D + 4Q_1^2e^{aQ_1}),
$$
$$
\left|\frac{1}{8}\int\limits_{0}^{2a}e^{izy}(q_2 - q_1)'\left(\frac{y}{2}\right) \,dy\right| \le \frac{1}{4}a^{\frac{1}{r}}D_r',
$$
$$
\left|\int\limits_{0}^{2a}e^{izy}\left(\frac{\partial ^2 H_{02}}{\partial  t^2} - \frac{\partial ^2 H_{01}}{\partial  t^2}\right)(0,y)\,dy\right| \le 2a\left( Q_1^3e^{aQ_1} + \frac{3}{8}Q_1D\right) + $$
$$+\frac{1}{4}\left|\int\limits_{0}^{2a}\left[ \rho_2h_2 - \rho_1h_1\right]\left(\frac{y}{2}\right) e^{izy}\,dy\right|\le 2a\left( Q_1^3e^{aQ_1} + \frac{3}{8}Q_1D\right)+ \frac{1}{2}aQ_1D+ Q_1^2 .
$$

\begin{lemma}
При произвольном $\beta \in (0, 1-\delta)$ и условии $R > R_5 = \max\left\{ a^{-\frac{1}{\beta}}, R_0, R_2\right\{$, где $R_0$ определено в теореме \ref{th_MainTheorem},  выполнена оценка
\begin{equation}
\begin{split}
&\left|\frac{d}{dt}\left( K_{02}(0, t) - K_{01}(0, t)\right)\right|\le
E_0\left[ \min\left( 1, \frac{3R^{-\beta}/\pi}{(2a-t)}\right) + \min\left( 1, \frac{3R^{-\beta}/\pi}{t}\right)\right] + \widehat{\Theta},
\end{split}
\end{equation}
где
$$
E_0 = \frac{D + 4Q_1^2e^{aQ_1}}{8}, \ \widehat{\Theta} = \Theta + \Theta_{\varepsilon},
$$
$$
\Theta = E_1R^{-\beta\frac{r-1}{r}}\ln\left(\frac{\pi a e}{3}R^{\beta}\right)^{\frac{r-1}{r}} + E_2R^{-\beta}\ln\left(\frac{\pi a e}{3}R^{\beta}\right) + E_3(1+\chi(R))R^{\beta-(1-\delta)\frac{p-1}{3p-2}} +
$$
$$
+E_4R^{-\beta} + R^{2\beta - 2(1 - \delta)},
$$
где
$$
E_1 = \frac{3}{2\pi}D_r', \ E_2 = \frac{6}{\pi}\left( Q_1^3e^{aQ_1}+\frac{3}{8}Q_1D\right), $$
$$ E_3 = \frac{5a}{\pi}Q_1C_0, \ E_4 = \frac{F(1 + Q_1e^{aQ_1})e^{2a(1+Q_1e^{aQ_1})}}{\pi},
$$
$$\Theta_{\varepsilon}=\frac{2}{\pi}e^{2a(1 + Q_1e^{aQ_1})}R^{\beta}\varphi_{\beta}(R, \varepsilon) + \frac{5a}{\pi} Q_1R^{\beta}\psi(R, \varepsilon).$$

\end{lemma}
{\bf Доказательство.} Оценка интеграла
$$
\int\limits_{|x| > \rho}x(\psi_2(x) - \psi_1(x))e^{-ixt}dx
$$
показывает, что он
сходится равномерно при $t\in [t_1, t_2]\subset (0, 2a)$, а это дает возможность дифференцирования под знаком интеграла:
\begin{equation}
\frac{d}{dt}(K_2(0, t) - K_1(0, t)) = \frac{-i}{2\pi} \left(\int\limits_{|x| \le R^{\alpha}} + \int\limits_{|x| > R^{\alpha}}\right)  x (\psi_2(x) - \psi_1(x))e^{-ixt}dx.
\end{equation}
Представление \eqref{eq_strong_representation} дает
$$
\int\limits_{|x| > \rho}x(\psi_2(x) - \psi_1(x))e^{-ixt}dx = \sum\limits_{k=1}^4 I_k(\rho, t),
$$
где
$$
I_1 = -\frac{h(2a)}{4}\int\limits_{|x| > \rho}\frac{1}{x}e^{-ix(t-2a)}dx, \ I_2 = \frac{h(0)}{4}\int\limits_{|x| > \rho}\frac{1}{x}e^{-ixt}dx,
$$
$$
I_3 = -\int\limits_{|x| > \rho} \frac{1}{x}e^{-ixt}\int\limits_0^{2a}e^{ixy}\left(\frac{\partial^2 H_{02}}{\partial t^2} - \frac{\partial^2 H_{01}}{\partial t^2}\right)(0, y)\,dy\,dx,$$
$$
I_4 = \frac{1}{8}\int\limits_{|x| > \rho}\frac{1}{x}e^{-ixt}\int\limits_0^{2a}e^{ixy}(q_2 - q_1)'\left(\frac{y}{2}\right)\,dy\,dx .
$$
Имеем
$$
|I_1| \le \frac{\pi (D + 4Q_1^2e^{aQ_1})}{2}\min\left( 1, \frac{3/\pi}{ \rho (2a - t)}\right), \ |I_2|\le \frac{\pi (D + 4Q_1^2e^{aQ_1})}{2}\min\left( 1, \frac{3/\pi}{ \rho t}\right),
$$
а при
$\rho > \frac{3}{\pi a }$
$$
|I_4|\le 3 D_r' \rho^{-\frac{r-1}{r}}\ln\left(\frac{\pi a e}{3}\rho\right)^{\frac{r-1}{r}}.
$$
Наконец, оценим $I_3$.
Имеем
$$
|I_3|\le 12 \left\Vert \left[ \frac{\partial ^2 H_{02}}{\partial  t^2} - \frac{\partial ^2 H_{01}}{\partial  t^2}\right](0, \cdot) + \frac{1}{4}\left[ \rho_2q_2 - \rho_1q_1\right]\left(\frac{\cdot}{2}\right)\right\Vert_{\infty}\cdot \rho^{-1}\cdot\ln\left(\frac{\pi a e}{3}\rho\right) +
$$
$$
+\frac{1}{4}\left|\int\limits_{|x|>\rho}e^{-ixt}\int\limits_{0}^{2a}e^{iyx}\left[ \rho_2^2 - \rho_1^2\right]\left(\frac{y}{2}\right) \,dy\,dx\right|;
$$
обозначим
$$
f(y) = \left[ \rho_2^2 - \rho_1^2\right]\left(\frac{y}{2}\right), \ y\in [0, 2a],
$$
тогда $f\in AC[0, 2a]$, $f(0) = f(a) = 0$;
продолжим $f$ нулем на $\mathbb{R}$. Имеем
$$
f(t) = \lim\limits_{\rho\to\infty}\frac{1}{2\pi}\int\limits_{-\rho}^{\rho}e^{-ixt}\int\limits_{\mathbb{R}}f(y)e^{iyx}\,dy\,dx,
$$
и согласно оценке \eqref{eq_MainEstimate}, при выполнении условий теоремы \ref{th_MainTheorem},
$$\Vert f\Vert_{\infty}\le 2Q_1 \left[\psi(R, \varepsilon) + C_0R^{-(1-\delta)\frac{p-1}{3p-2}}(1 + \chi(R))\right];$$ далее,
$$
\int\limits_{|x|>\rho}e^{-ixt}\int\limits_{\mathbb{R}}f(y)e^{iyx}\,dy\,dx  = 2\pi f(t)-\frac{1}{2\pi}\int\limits_{-\rho}^{\rho}e^{-ixt}\int\limits_{\mathbb{R}}f(y)e^{iyx}\,dy\,dx  =
$$
$$=-2\left[ \int\limits_{-\infty}^{t-2a} + \int\limits_{t-2a}^{t+2a} + \int\limits_{t+2a}^{\infty}\right]\frac{\sin(\rho(y-t))}{y-t}[f(y) - f(t)]dy,
$$
$$
\left|\int\limits_{t+2a}^{\infty}\frac{\sin(\rho(y-t))}{y-t}[f(y) - f(t)]dy\right| = \left| f(t)\right| \left|\int\limits_{2a}^{\infty}\frac{\sin(\rho s)}{s}\,ds\right| \le \Vert f\Vert_{\infty}\frac{\pi}{2}\min\left( 1, \frac{2/\pi}{a\rho}\right),
$$
такая же оценка для
$
\left|\int\limits_{-\infty}^{t-2a}\frac{\sin(\rho(y-t))}{y-t}[f(y) - f(t)]dy\right|$ и
$$
\left|\int\limits_{t-2a}^{t+2a}\frac{\sin(\rho(y-t))}{y-t}[f(y) - f(t)]\,dy\right|\le 8a\rho \Vert f\Vert_{\infty};
$$
складывая эти оценки, получаем при  $\rho > \frac{1}{a}$
$$
\left|\int\limits_{|x|>\rho}e^{-ixt}\int\limits_{\mathbb{R}}f(y)e^{iyx}\,dy\,dx \right| \le 20 \Vert f\Vert_{\infty}a\rho.
$$

Наконец, положим $\rho = R^{\beta}$ и оценим интеграл по контуру $\Gamma(A, R^{\beta})$, где
$$
A:=1 + Q_1e^{aQ_1}.
$$
Имеем, согласно оценке \eqref{eq_strong_estimate},
$$
\frac{1}{2\pi}\left|\int\limits_{\pm R^{\beta}}^{\pm R^{\beta} + Ai}-iz(\psi_2(z) - \psi_1(z))e^{-izt}dz\right|\le \frac{AFe^{2aA}}{2\pi}R^{-\beta}.
$$

Согласно лемме \ref{lemma_main_analytic}, при $R > R_0$ и $A < R^{\beta} < R^{1-\delta}$ для $z\in [-R^{\beta} + Ai, R^{\beta} + Ai]$ справедлива оценка
$$
\left| 1 - e^{g_2(z) - g_1(z)}\frac{\Pi_2(R,z)}{\Pi_1(R, z)}\right|\le \eta R^{-2(1-\delta)};
$$
значит, как в доказательстве леммы \ref{lemma_K_2_minus_K_1_estimate}, полагая $\eta = \frac{\pi}{2}e^{-2aA}$, получаем
$$
\left|\frac{1}{2\pi}\int\limits_{-R^{\beta} + Ai}^{R^{\beta} + Ai}-iz(\psi_2(z) - \psi_1(z))e^{-izt} \, dz\right|\le R^{ \beta- 2(1 - \delta)} + \frac{2}{\pi}e^{2aA}R^{2\beta} \varphi_{\beta}(R, \varepsilon),
$$
$$
\frac{1}{2\pi}\left|\int\limits_{-R^{\beta}}^{R^{\beta}}-ix(\psi_2(x) - \psi_1(x))e^{-ixt}dx\right| \le  \frac{AF}{\pi}e^{2aA}R^{-\beta} + R^{2\beta-2(1-\delta)} + \frac{2}{\pi}e^{2aA}R^{\beta} \varphi_{\beta}(R, \varepsilon),$$
где $\varphi_{\beta}$ определено  в \eqref{eq_varphi}.

Используя \eqref{th_MainTheorem}, получаем требуемое при $R > R_5:=\max\left\{ a^{-\frac{1}{\beta}}, R_0, R_2\right\}$.

\begin{cor}
Существуют константы $R_6$ и $F_0$, зависящее от априорных параметров, такие, что при условии $R > R_6$ выполнено
\begin{equation}
\label{eq_optimized}
\begin{split}
&\left| \frac{d}{dt}\left( K_{02}(0, t) - K_{01}(0, t)\right)\right| \le E_0 \left[ \min\left( 1, \frac{3R^{-\beta_*}/\pi}{(2a-t)}\right) + \min\left( 1, \frac{3R^{-\beta_*}/\pi}{t}\right)\right] + \\
&+F_0R^{-\beta^{*}\frac{r-1}{r}}(1 + \lambda(R)),
\end{split}
\end{equation}
где
$$
\beta_* = (1-\delta)\frac{r}{2r-1}\frac{p-1}{3p-2}, \ \lambda(R) = O(\ln (R)), \ R\to\infty.
$$
\end{cor}

{\bf Доказательство теоремы \ref{th_SecondTheorem}.}
Имеем
$$
\frac{\partial  K_{12}}{\partial  t}(0, t) = \frac{d}{dt}\left[ K_{02} - K_{01}\right](0, t) - \left[ K_{02} - K_{01}\right](0, t)\rho_1(t) + $$
$$+\frac{1}{2}\int\limits_{0}^{t}\frac{d}{dy}\left[ K_{02} - K_{01}\right](0, y)\rho_{1}\left(\frac{y+t}{2}\right) dy
$$
и
$$
q_2(x) - q_1(x) = -2\int\limits_{a}^{x}\left( q_1(y + x) - q_2(y-x)\right) \widetilde{K}_{12}(x, y)dy +
$$
$$
+2\int\limits_{0}^{x}\left( q_1(x + y) - q_2(x - y)\right)\widetilde{K}_{12}(y, x)dy - 4\frac{\partial K_{12}}{\partial t}(0, 2x).
$$

При оценке этого выражения с помощью предыдущих неравенств возникают слагаемые, в которые $R$ входит в степени, не превосходящей $-\beta^*\frac{r-1}{r} = -(1-\delta)\frac{r}{2r-1}\frac{p-1}{3p-2}$.

Address:

Lomonosov Moscow State University, Department of Mechanics and Mathematics.

email:  valgeynts@gmail.com;  shkalikov@mi.ras.ru

\end{document}